\documentclass[12pt,reqno]{amsart}
\usepackage{amsfonts}
\usepackage{amsmath}
\usepackage{amsthm}
\usepackage{mathrsfs}
\usepackage{stmaryrd}
\usepackage{graphicx}
\usepackage{epstopdf}
\usepackage{textcomp}

\newtheorem{theorem}{Theorem}
\newtheorem{lemma}{Lemma}
\newcommand{\Tr}{{\rm Tr\,}}
\theoremstyle{definition}
\newtheorem{definition}{Definition}
\newtheorem{assumption}{Assumption}
\newtheorem{remark}{Remark}
\begin{document}
\title[Regular finite fuel stochastic control problems]{Regular finite fuel stochastic control problems with exit time}

\author{Dmitry B. Rokhlin}
\author{Georgii Mironenko}

\address[D.B. Rokhlin and G. Mironenko]{Institute of Mathematics, Mechanics and Computer Sciences,
              Southern Federal University,
Mil'chakova str., 8a, 344090, Rostov-on-Don, Russia}
\email[D.B. Rokhlin]{rokhlin@math.rsu.ru}
\email[G. Mironenko]{georim89@mail.ru}
\thanks{The research of the D.B.\,Rokhlin is supported by Southern Federal University, project 213.01-07-2014/07.}

\begin{abstract} We consider a class of exit time stochastic control problems for diffusion processes with discounted criterion, where the controller can utilize a given amount of resource, called "fuel". In contrast to the vast majority of existing literature, concerning the "finite fuel" problems, it is assumed that the intensity of fuel consumption is bounded. We characterize the value function of the optimization problem as the unique continuous viscosity solution of the Dirichlet boundary value problem for the correspondent Hamilton-Jacobi-Bellman (HJB) equation. Our assumptions concern the HJB equations, related to the problems with infinite fuel and without fuel. Also, we present computer experiments, for the problems of optimal regulation and optimal tracking of a simple stochastic system with the stable or unstable equilibrium point.
\end{abstract}

\maketitle

\section{Introduction} 
\setcounter{equation}{0}

Consider a controller, whose aim is to keep a stochastic system $X$ in a prescribed domain $G$ as long as possible. An influence on the system requires the resource (or fuel) expenditure. The problem is to utilize the given resource amount in an optimal way.

The study of diffusion stochastic control problems with bounded amount of fuel was initiated in \cite{BatChe67b}. The problem was to bring a space-ship close to a stochastic target. Afterwards, the "finite fuel" issues were developed, almost exclusively, in the singular stochastic control paradigm of \cite{BatChe67a,BatChe67b}. In this framework it is always optimal to keep the system in a "no-action region". An essential step was made in \cite{BenSheWit1980}. This paper investigated the problem of optimal tracking of a Wiener process by a process, whose variation is bounded by a given initial fuel amount. An explicit solution was obtained for quadratic discounted criterion in the infinite horizon case. In particular, it was mentioned that the optimal no-actions region becomes wider as the available fuel amount decreases.

More general problems were then studied, e.g., in \cite{KarShr86,BriShr92,KarOcoWanZer00}. The case of finite horizon was considered in \cite{Kar85,ChoMenRob85,ElKKar88}. Quite general results, concerning the characterization of the value function and the existence of optimal control strategies were obtained in \cite{DufMil02,MotSar07,MotSar08a,MotSar08b}. A lot of studies were motivated by applications. We mention the problems of optimal correction of motion \cite{Che71}, controlling a satellite \cite{Jac83,Jac99,Jac02}, reaching the goal by a player \cite{SudWee92,Wee92}, optimal liquidation and trade execution \cite{PemZhaYin07,GatSch11,BiaWuZhe14,BecBilFre15}.

In contrast to the vast majority of existing literature, we assume that the resource (or fuel) intensity consumption is bounded.
In fact, only \cite{PemZhaYin07,BiaWuZhe14} from the above references, adopt the same assumption. Confining ourselves to classical controls, we deal with simpler mathematical problems which still can present interesting effects and have a wide range of applications.

To give the precise formulation of our problem consider a standard $m$-dimensional Wiener process $W=(W^1,\dots W^m)$ defined on a
probability space $(\Omega,\mathcal F,\mathsf P)$. Denote by $\mathbb F=(\mathscr F_s)_{s\ge 0}$ the minimal augmented natural filtration of $W$. Let the controlled process $X=(X^1,\dots,X^d)$ be governed by the system of stochastic differential equations
\begin{equation} \label{eq:1.1}
dX_t=b(X_t,\alpha_t) dt+\sigma(X_t,\alpha_t) dW_t, \quad X_0=x.
\end{equation}
An $\mathbb F$-progressively measurable process $\alpha \in A=[\underline a, \overline a]$, $\underline a\le 0\le \overline a$ is regarded as the intensity of resource consumption. We assume that the components of the drift vector $b:\mathbb R^d\times A\mapsto\mathbb R^d$ and the diffusion matrix $\sigma:\mathbb R^d\times A\mapsto\mathbb R^d\times\mathbb R^m$ are continuous, and
$$ |b(x,a)-b(y,a)|+|\sigma(x,a)-\sigma(y,a)|\le K|x-y|$$
with some constant $K$ independent of $a$. Note that this inequality implies also the linear growth condition:
$$ |b(x,a)|+|\sigma(x,a)|\le K'(1+|x|).$$
Thus, there exist a unique $\mathbb F$-adapted strong solution of (\ref{eq:1.1}) on $[0,\infty)$: see \cite{Kry80} (Chap. 2, Sect. 5). The resource amount $Y$ satisfies the equation
\begin{equation} \label{eq:1.2}
dY_t=-|\alpha_t|dt,\quad Y_0=y.
\end{equation}

We call a process $\alpha$ \emph{admissible}, and write $\alpha \in \mathcal A(x,y)$, if $Y_t \ge 0$, $t\ge 0$. The admissibility means that the resource overrun is prohibited. The solution of (\ref{eq:1.1}), (\ref{eq:1.2}) is denoted by $X^{x,y,\alpha}, Y^{x,y,\alpha}$.

Let $G\subset \mathbb R^d$ be an open set. It will be convenient to assume that $0\in G$. In general, we do not require $G$ to be bounded. Denote by
$\theta^{x,y,\alpha}=\inf\{t\ge 0:X_t^{x,y,\alpha}\not \in G\}$ the exit time of $X^{x,y,\alpha}$ from $G$.
The objective functional $J$ and the value function $v$ are defined by
\begin{equation} \label{eq:1.3}
J(x,y,\alpha)=\mathsf E\int_0^{\theta^{x,y,\alpha}} e^{-\beta t} f(X_t^{x,y,\alpha},\alpha_t)\,dt, \qquad
   v(x,y)=\sup_{\alpha\in\mathcal A(x,y)} J(x,y,\alpha),
\end{equation}
where $\beta>0$, and $f:\mathbb R^d\times A\mapsto\mathbb R$ is a bounded continuous function. Note that for $f=1$ we obtain the risk-sensitive criterion
\begin{equation} \label{eq:1.4}
J(x,y,\alpha)=\frac{1}{\beta}\left(1-\mathsf E e^{-\beta \theta^{x,y,\alpha}}\right),
\end{equation}
related to the maximization of the expected time $\mathsf E \theta^{x,y,\alpha}$ before leaving $G$.
The minimization of $\mathsf E e^{-\beta \theta^{x,y,\alpha}}$, as compared to the maximization of $\mathsf E \theta^{x,y,\alpha}$, produces the controls for which the probability of an early exit is smaller (see \cite{DupMcE97,ClaVin12}).

As is mentioned above, suitably formulated problems of this sort appear in various applications. We indicate one more example, which motivated the present work. Let $X$ describe the exchange rate between a domestic and a foreign currency. The controller is a central bank, trying to support the national currency and, thus, to hold $X$ above the level $l$ (an actual problem for Russian Ruble). 
The available amount of the foreign currency to be sold in the market is $Y$. Mathematically this problem is quite similar to that of \cite{Jac83}.
However, the use of regular controls (with bounded intervention intensity) may result in exit of $X$ from $G=(l,\infty)$ before the currency resource (the "fuel") $Y$ is exhausted. Within our model, the central bank choose to stop interventions at the occurrence of the stopping time $\theta$, regarding it as a signal to reduce the level $l$ or to change the set of monetary policies. It should be noted that there is an extensive literature concerning the exchange rates regulation. Without going into further discussion, we only mention that the prevailing paradigm is to model currency interventions by impulse controls: see, e.g., \cite{ÑadZap00,BenLonPerSet12} and references therein.

The present paper is organized as follows. In Section \ref{sec:2} we prove that the value function (\ref{eq:1.3}) is the unique continuous viscosity solution of the correspondent Hamilton-Jacobi-Bellman (HJB) equation  in the half-cylinder $G\times (0,\infty)$ with Dirichlet boundary conditions: see Theorem \ref{th:1}. The proof is based on the stochastic Perron method \cite{BaySir13}, adapted to exit time problems in \cite{Rok14}, and does not appeal to the dynamic programming principle. Our Assumptions \ref{as:1}-\ref{as:3} are of analytical nature. They concern the well-posedness of the boundary value problems for the HJB equations, related to the problems with infinite fuel and without fuel, as well as some nice properties of the solution of the latter.

Theorem \ref{th:1} allows to justify the convergence of appropriate finite difference schemes. In Sections \ref{sec:3} and \ref{sec:4} we present computer experiments, for the problems of optimal regulation and optimal tracking of a simple stochastic system with the stable or unstable equilibrium point.

\section{Characterization of the value function} \label{sec:2}
\setcounter{equation}{0}
For an open set $\mathcal O\subset\mathbb R^d$ consider the differential equation
\begin{equation} \label{eq:2.1}
F(x,u(x),u_x(x),u_{xx}(x))=0,\quad x\in\mathcal O,
\end{equation}
where $u_x$ is the gradient vector and $u_{xx}$ is the Hessian matrix. It is assumed that the function $F:\mathcal O\times\mathbb R\times\mathbb R^n\times\mathbb S^n$ is continuous and satisfies the monotonicity property:
$$ F(x,r,p,X)\le F(x,s,p,Y) \quad \textrm{whenever } r\le s\ \textrm{and } X-Y\ \textrm{is positive semidefnite}.$$
Here $\mathbb S^n$ is the set of symmetric $n\times n$ matrices.

Let $g$ be a continuous function on $\partial\mathcal O $. We consider two variants of the boundary conditions:
\begin{equation} \label{eq:2.2}
u=g\ \textrm{on } \partial\mathcal O ,
\end{equation}
\begin{equation} \label{eq:2.3}
u=g\ \textrm{or } F(x,u,u_x,u_{xx})=0\ \textrm{on } \partial\mathcal O .
\end{equation}

The equation (\ref{eq:2.1}), as well as the boundary conditions (\ref{eq:2.2}), (\ref{eq:2.3}), should be understood in the viscosity sense (the classical reference is \cite{CraIshLio92}). Recall that a bounded upper semicontinuous (usc) function $u$ is called a \emph{viscosity subsolution} of the equation (\ref{eq:2.1}) and the boundary condition (\ref{eq:2.3}) (resp., (\ref{eq:2.2})) if for any $z\in\overline{\mathcal O}$, $\varphi\in C^2(\mathbb R^n)$, such that $z$ is a local maximum point of $u-\varphi$ in $\overline{\mathcal O}$, we have
$$F(z,u(z),\varphi_x(z),\varphi_{xx}(z))\le 0\quad \textrm{if } z\in\mathcal O,$$
$$ u(z)\le g(z)\ \textrm{or } F(z,u(z),\varphi_x(z),\varphi_{xx}(z))\le 0\quad \textrm{if } z\in\partial\mathcal O,$$
$$ (\textrm{resp., } u(z)\le g(z)\quad \textrm{if }\ z\in\partial\mathcal O).$$

The notion of a \emph{viscosity supersolution} is defined symmetrically. One should consider a bounded lower semicontinuous (lsc) function $u$ and, assuming that $z$ is a local minimum point of $u-\varphi$ in $\overline{\mathcal O}$, postulate the reverse inequalities.

A bounded function $u$ is called a \emph{viscosity solution} of (\ref{eq:2.1}), (\ref{eq:2.3}) (or (\ref{eq:2.1}), (\ref{eq:2.2})) if its usc and lsc envelopes:
$$ u^*(x)=\inf_{\varepsilon>0}\sup\{ u(y): y\in B_\varepsilon(x)\cap\overline {\mathcal O}\},\ \
   u_*(x)=\sup_{\varepsilon>0}\inf\{ u(y): y\in B_\varepsilon(x)\cap\overline {\mathcal O}\}$$
are respectively viscosity sub- and supersolutions of these equations. Here $B_\varepsilon(x)$ is the open ball in $\mathbb R^n$ centered at $x$ with radius $\varepsilon$. Following \cite{BarSou91}, we say that (\ref{eq:2.1}), (\ref{eq:2.3}) satisfies the \emph{strong uniqueness property}, if for any viscosity sub- and supersolutions $u$, $w$ of (\ref{eq:2.1}), (\ref{eq:2.3}) we have $u\le w$ on $\overline{\mathcal O}$.

Consider the family $\mathcal L^a$ of "infinitesimal generators" of the diffusion process $X$:
$$ \mathcal L^a\varphi(x)=b(x,a)\varphi_x(x)+ \frac{1}{2}\Tr(\sigma(x,a)\sigma^T(x,a)\varphi_{xx}(x))$$
and the Hamilton-Jacobi-Bellman (HJB) equation, related to the problem (\ref{eq:1.1})-(\ref{eq:1.3}):
\begin{equation} \label{eq:2.4}
\beta v- H(x,v_x,v_y,v_{xx})=0,\quad (x,y)\in \Pi:=G\times (0,\infty),
\end{equation}
$$H(x,v_x,v_y,v_{xx})=\sup_{a \in [\underline a, \overline a]} \left\{f(x,a)+ \mathcal L^a v-|a|v_y \right\}.$$
The aim of this section is to prove that the value function (\ref{eq:1.3}) is the unique continuous viscosity solution of (\ref{eq:2.4}) with appropriate boundary conditions (see Theorem \ref{th:1}). This requires some preliminary work.

Denote by $C^2(G)$ the set of two times continuously differentiable functions on $G$, and by $C_b(\overline G)$ the set of bounded continuous functions on $\overline G$. Put $\widehat f(x)=f(x,0)$.

\begin{assumption} \label{as:1} There exists a solution $\psi\in C_b(\overline G)\cap C^2(G)$ of the Dirichlet problem
\begin{equation} \label{eq:2.5}
\beta\psi(x)-\widehat f(x)-\mathcal L^0\psi (x)=0,  \ x\in G; \quad \psi=0 \ \text{on}\ \partial G.
\end{equation}
\end{assumption}

It is straightforward to show that such solution is unique and admits the probabilistic representation
\begin{align*}
\psi(x) &=\mathsf E\int_0^{\widehat\theta^x}e^{-\beta t} \widehat f(\widehat X_t^x)\,dt,\qquad \widehat\theta^x=\inf\{t\ge: \widehat X_t\not\in G\},\\
d\widehat X_t^x &=b(\widehat X_t^x,0)\,dt+\sigma(\widehat X_t^x,0)\,dW_t,\quad \widehat X_t=x
\end{align*}
(see \cite{Fre85}, Chap. II, Theorem 2.1 and Remark 1 after it). Note, that $\psi$ is the value function of the problem without fuel.

One can see that (\ref{eq:1.1})-(\ref{eq:1.3}) combines the features of exit time and state constrained control problems. Fortunately, it is equivalent to a pure the exit time problem. Let $T^{x,y,\alpha}$ be the exit time of $(X^{x,y,\alpha}, Y^{x,y,\alpha})$ from the open set $G\times (0,\infty)$. For stopping times $\tau\le\sigma$ with values in $[0,\infty]$ we denote by $\llbracket\tau,\sigma\rrbracket$ the stochastic interval $\{(\omega,t)\in\Omega\times [0,\infty):\tau(\omega)\le t\le\sigma(\omega)\}$.
\begin{lemma} \label{lem:1}
Under Assumption \ref{as:1} the value function (\ref{eq:1.3}) admits the representation
\begin{equation} \label{eq:2.6}
 v(x,y)=\sup_{\alpha\in\mathcal U}\mathsf E\left(\int_0^{T^{x,y,\alpha}} e^{-\beta t} f(X_t^{x,y,\alpha})\,dt+e^{-\beta T^{x,y,\alpha}}\psi(X_{T^{x,y,\alpha}}^{x,y,\alpha})\right),
\end{equation}
where $\mathcal U$ is the set of all $\mathbb F$-progressively measurable strategies $\alpha$ with values in $A$.
\end{lemma}
\begin{proof} Note, that
$T^{x,y,\alpha}=\theta^{x,y,\alpha}\wedge\inf\{t\ge 0:Y_t^{x,y,\alpha}=0\},$
and any admissible strategy $\alpha\in\mathcal A(x,y)$ should be switched to $0$ as far as "the fuel $Y^{x,y,\alpha}$ is exhausted": $\alpha_t=0$, $t\in (T^{x,y,\alpha},\theta^{x,y,\alpha}]$. Hence, the functional (\ref{eq:1.3}) can be represented as
\begin{align} \label{eq:2.7}
J(x,y,\alpha)&=\mathsf E\int_0^{T^{x,y,\alpha}} e^{-\beta t} f(X_t^{x,y,\alpha},\alpha_t)\,dt \nonumber\\
&+\mathsf E\left(I_{\{T^{x,y,\alpha}<\theta^{x,y,\alpha}\}}\mathsf E\left(\int_{T^{x,y,\alpha}}^{\theta^{x,y,\alpha}} e^{-\beta t} \widehat f(X_t^{x,y,\alpha}) \,dt\biggr|\mathscr F_{T^{x,y,\alpha}}\right)\right).
\end{align}

On the stochastic interval $\llbracket T^{x,y,\alpha}, \theta^{x,y,\alpha}\llbracket$ we have
$$ X^{x,y,\alpha}_t=\xi+\int_{T^{x,y,\alpha}}^t b(X^{x,y,\alpha}_s,0)\,ds+\int_{T^{x,y,\alpha}}^t \sigma(X^{x,y,\alpha}_s,0)\,dW_s,
\quad \xi=X^{x,y,\alpha}_{T^{x,y,\alpha}}.$$
For the solution $\psi$ of the Dirichet problem (\ref{eq:2.5}) by Ito's formula we get
\begin{align} \label{eq:2.8}
& e^{-\beta t}\psi(X^{x,y,\alpha}_t) =e^{-\beta T^{x,y,\alpha}}\psi(\xi)+\int_{T^{x,y,\alpha}}^t e^{-\beta s} (\mathcal L^0\psi-\beta\psi)(X^{x,y,\alpha}_s)\,ds+M_t \nonumber\\
 &=e^{-\beta T^{x,y,\alpha}}\psi(\xi)-\int_{T^{x,y,\alpha}}^t e^{-\beta s}\widehat f(X^{x,y,\alpha}_s)\,ds+M_t\quad\textrm{on } \llbracket T^{x,y,\alpha}, \theta^{x,y,\alpha}\llbracket,
\end{align}
where $M_t=\int^t_{T^{x,y,\alpha}} e^{-\beta s} \psi_x(X^{x,y,\alpha}_s) \cdot \sigma(X^{x,y,\alpha}_s,0) \, dW_s$ is a local martingale. The last equality shows, however, that $M$ is bounded. Hence, $M$ is a uniformly integrable continuous martingale with $M_{T^{x,y,\alpha}}=0$ on $\{T^{x,y,\alpha}<\theta^{x,y,\alpha}\}$. For any stopping time $\tau$, such that
\begin{equation} \label{eq:2.9}
T^{x,y,\alpha}\le\tau<\theta^{x,y,\alpha}\quad \text{on } \{T^{x,y,\alpha}<\theta^{x,y,\alpha}\},
\end{equation}
by taking the conditional expectation, from (\ref{eq:2.8}) we obtain
\begin{align} \label{eq:2.10}
& \mathsf E\left(e^{-\beta\tau}\psi(X^{x,y,\alpha}_\tau)I_{\{T^{x,y,\alpha}<\theta^{x,y,\alpha}\}}\Bigr|\mathscr F_{T^{x,y,\alpha}}\right)=e^{-\beta T^{x,y,\alpha}}\psi(\xi) I_{\{T^{x,y,\alpha}<\theta^{x,y,\alpha}\}} \nonumber\\
& -I_{\{T^{x,y,\alpha}<\theta^{x,y,\alpha}\}} \mathsf E\left(\int_{T^{x,y,\alpha}}^\tau e^{-\beta s}\widehat f(X^{x,y,\alpha}_s)\,ds\Bigr|\mathscr F_{T^{x,y,\alpha}}\right)
\end{align}

Consider an expanding sequence $G_n$ of compact sets such that $\cup_{n\ge 1} G_n=G$ and put
$$\tau_n=T^{x,y,\alpha}\vee \inf\{t\ge 0:X_t^{x,y,\alpha}\not \in G_n\}.$$
Clearly, $\tau_n\nearrow\theta^{x,y,\alpha}$ and  $\tau_n$ satisfy the condition (\ref{eq:2.9}) imposed on $\tau$. Furthermore,
$$ \lim_{n\to\infty} e^{-\beta\tau_n}\psi(X^{x,y,\alpha}_{\tau_n}) I_{\{T^{x,y,\alpha}<\theta^{x,y,\alpha}\}}=0\quad \text{a.s.}$$
by the boundary condition (\ref{eq:2.5}) and the boundedness of $\psi$.
By the inequality
$$\mathsf E\left|\mathsf E\left(e^{-\beta\tau_n}\psi(X^{x,y,\alpha}_{\tau_n}) I_{\{T^{x,y,\alpha}<\theta^{x,y,\alpha}\}}\Bigr|\mathscr F_{T^{x,y,\alpha}}\right) \right|\le \mathsf E\left| e^{-\beta\tau_n}\psi(X^{x,y,\alpha}_{\tau_n}) I_{\{T^{x,y,\alpha}<\theta^{x,y,\alpha}\}}\right|$$
and the dominated convergence theorem it follows that
$$ \mathsf E\left(e^{-\beta\tau_n}\psi(X^{x,y,\alpha}_{\tau_n}) I_{\{T^{x,y,\alpha}<\theta^{x,y,\alpha}\}}\Bigr|\mathscr F_{T^{x,y,\alpha}}\right)\to 0\quad \text{in } L^1.$$
Passing to a subsequence, one may assume that this sequence converges with probability $1$. Then from (\ref{eq:2.10}) we get
$$ I_{\{T^{x,y,\alpha}<\theta^{x,y,\alpha}\}}\mathsf E\left(\int_{T^{x,y,\alpha}}^{\theta^{x,y,\alpha}} e^{-\beta s}\widehat f(X_s^{x,y,\alpha})\,ds\biggr|\mathscr F_{T^{x,y,\alpha}}\right)=I_{\{T^{x,y,\alpha}<\theta^{x,y,\alpha}\}} e^{-\beta T^{x,y,\alpha}}\psi(\xi).$$

This completes the proof, since (\ref{eq:2.7}) takes the form
$$ J(x,y,\alpha)=\mathsf E\left(\int_0^{T^{x,y,\alpha}} e^{-\beta t} f(X_t^{x,y,\alpha},\alpha_t)\,dt+I_{\{T^{x,y,\alpha}<\theta^{x,y,\alpha}\}} e^{-\beta T^{x,y,\alpha}}\psi(X_{T^{x,y,\alpha}}^{x,y,\alpha})\right),$$
which is equivalent to the representation (\ref{eq:2.6}) in view of the boundary condition (\ref{eq:2.5}).
\end{proof}

Denote by $\overline v$ the value function of the problem with "infinite fuel":
\begin{equation} \label{eq:2.11}
\overline v(x)=\sup_{\alpha\in\mathcal U} \mathsf E\int_0^{\theta^{x,\alpha}} e^{-\beta t} f(X_t^{x,\alpha},\alpha_t)\,dt,
\qquad \theta^{x,\alpha}=\inf\{t\ge 0: X_t^{x,\alpha}\not \in G\},
\end{equation}
where $X^{x,\alpha}$ is the solution of (\ref{eq:1.1}). Consider the correspondent HJB equation and the boundary conditions:
\begin{equation} \label{eq:2.12}
\beta \overline v- \sup_{a \in [\underline a, \overline a]} \left\{f(x,a)+ \mathcal L^a \overline v \right\}=0,\quad x\in G,
\end{equation}
\begin{equation} \label{eq:2.13}
\beta \overline v- \sup_{a \in [\underline a, \overline a]} \left\{f(x,a)+ \mathcal L^a \overline v \right\}=0 \quad
\text{or}\quad \overline v=0\quad \text{on } \partial G.
\end{equation}

\begin{assumption} \label{as:2}
The boundary value problem (\ref{eq:2.12}), (\ref{eq:2.13}) satisfies the strong uniqueness property.
\end{assumption}

Let us mention a simple sufficient condition, ensuring the validity of Assumption \ref{as:2}. Suppose that $\partial G$ is of class $C^2$ and denote by $n(x)$ the unit outer normal to $\partial G$ at $x$. The Assumption \ref{as:2} holds true if the diffusion matrix does not degenerate along the normal direction to the boundary:
\begin{equation} \label{eq:2.13A}
\sigma(x,a) n(x) \neq 0, \quad (x,a) \in \partial G \times A.
\end{equation}
Indeed, let $u$, $w$ be bounded viscosity sub- and supersolution of (\ref{eq:2.12}), (\ref{eq:2.13}).  By Proposition 4.1 of \cite{BarRou98}, the generalized Dirichlet boundary condition (\ref{eq:2.13}) is satisfied in the usual sense: $u\le 0\le w$ on $\partial G$. Thus, we can apply the comparison result \cite[Theorem 7.3]{Ish89}, \cite[Theorem 4.2]{MotSar08a} (for not necessary bounded domain $G$) to get the inequality $u\le w$ on $\overline G$.

\begin{assumption} \label{as:3}
There exists a constant $K>0$ such that
$$ \sup_{x\in G}\left\{|f(x,a)-f(x,0)+\mathcal L^a\psi(x)-\mathcal L^0\psi(x)|\right\}\le K|a|.$$
\end{assumption}

This assumption is satisfied, e.g., if  $f$ is Lipshitz continuous, $\mathcal L^0$ is strictly elliptic and $\partial G$ is a bounded domain of H\"{o}lder class $C^{2,\alpha}$, $\alpha>0$. Indeed, by the classical Shauder theory we have $\psi\in C^{2,\alpha}(\overline G)$ (see \cite{GilTru01}, Theorem 6.14) and, in particular, the derivatives of $\psi$ up to second order are uniformly bounded in $G$. Note, that Assumption 1 also holds true in this case.

Define a continuous function $g$ on $\partial\Pi$ by the formulas
\begin{equation} \label{eq:2.14}
g(0,x)=\psi(x),\ \ x\in\overline G;\quad g(x,y)=0,\ x\in\partial G,\ y\ge 0.
\end{equation}
\begin{theorem} \label{th:1}
 Under Assumptions 1-3 the value function $v$ is a unique bounded viscosity solution of (\ref{eq:2.4}), which is continuous on $\overline\Pi$ and satisfies the boundary condition
\begin{equation} \label{eq:2.15}
v=g\quad \textnormal{on\ } \partial\Pi.
\end{equation}
\end{theorem}

A specific feature of the equation (\ref{eq:2.4}) is the form at which it degenerates at the boundary points $(x,0)$. Such degeneracy does not allow to apply directly the strong comparison result of \cite[Theoorem 2.1]{BarRou98}, \cite[Theoorem 2.1]{Cha04}. It is possible to apply the results of \cite{MotSar08a} after some additional work. We, however, pursue another away,
utilizing the stochastic Perron method, developed in \cite{BaySir13}. Using the result of \cite{Rok14}, this allows to give a short direct proof of Theorem \ref{th:1} without relying on the dynamic programming principle.

Let $\tau$ be a stopping time, and let $(\xi,\eta)$ be a bounded $\mathscr F_\tau$-measurable random vector with values in $\overline\Pi$. Consider the SDE (\ref{eq:1.1}) with the randomized initial condition $(\tau,\xi,\eta)$:
\begin{align}
X_t &=\xi I_{\{t\ge\tau\}}+\int_\tau^t b(X_s,\alpha_s)\,ds+\int_\tau^t \sigma(X_s,\alpha_s)\,dW_s, \label{eq:2.16}\\
Y_t &=\eta I_{\{t\ge\tau\}}-\int_\tau^t|\alpha_s|\,ds. \label{eq:2.17}
\end{align}
As is known, see \cite{Kry80} (Chap. 2, Sect. 5), there exists a pathwise unique strong solution  $(X^{\tau,\xi,\eta,\alpha},Y^{\tau,\xi,\eta,\alpha})$ of (\ref{eq:2.16}), (\ref{eq:2.17}) for any $\alpha\in\mathcal U$. To reconcile this notation with the previous one we drop the index $\tau$ for $\tau=0$: $X^{0,x,y,\alpha}=X^{x,y,\alpha}$ for instance.

For a continuous function $u$ on $\overline\Pi$ define the process
$$ Z^{\tau,\xi,\eta,\alpha}_t(u)=\int_\tau^t e^{-\beta s} f(X^{\tau,\xi,\eta,\alpha}_s,\alpha_s)\,ds +e^{-\beta t} u(X^{\tau,\xi,\eta,\alpha}_t,Y^{\tau,\xi,\eta,\alpha}_t).$$

\begin{definition} \label{def:1} A function $u\in C(\overline\Pi)$, such that $u(\cdot,y)$ is bounded, is called a stochastic subsolution of (\ref{eq:2.4}), (\ref{eq:2.15}), if $u\le g$ on $\partial\Pi$ and for any randomized initial condition $(\tau,\xi,\eta)$ there exists $\alpha\in\mathcal U$ such that
$$ \mathsf E(Z^{\tau,\xi,\eta,\alpha}_\rho(u)|\mathscr F_\tau)\ge Z^{\tau,\xi,\eta,\alpha}_\tau(u)=e^{-\beta\tau} u(\xi,\eta)$$
for any stopping time $\rho\in [\tau,T^{\tau,\xi,\eta,\alpha}]$.
\end{definition}

\begin{definition} \label{def:2} A function $w\in C(\overline\Pi)$, such that $w(\cdot,y)$ is bounded, is called a stochastic supersolution of (\ref{eq:2.4}), (\ref{eq:2.15}), if $w\ge g$ on $\partial\Pi$ and
$$ \mathsf E(Z^{\tau,\xi,\eta,\alpha}_\rho(w)|\mathscr F_\tau)\le Z^{\tau,\xi,\eta,\alpha}_\tau(w)=e^{-\beta\tau} w(\xi,\eta)$$
for any randomized initial condition $(\tau,\xi,\eta)$, control process $\alpha\in\mathcal U$ and stopping time $\rho\in [\tau,T^{\tau,\xi,\eta,\alpha}]$.
\end{definition}

In this form the notions of stochastic semisolutions are tailor-made for exit time problems (see \cite{Rok14}). In the present context we need not assume that $u$, $w$ are bounded in $y$, since for any bounded initial condition $(\xi,\eta)$ the process $Y^{\tau,\xi,\eta,\alpha}$ remains bounded up to the exit time $T^{\tau,\xi,\eta,\alpha}$ from $\Pi$. Note that $(X^{\tau,\xi,\eta,\alpha}_\infty,Y^{\tau,\xi,\eta,\alpha}_\infty)$ may be defined arbitrary.

Denote the sets of stochastic sub- and supersolutions by $\mathcal V_-$ and $\mathcal V_+$ respectively. Any stochastic subsolution $u$ bounds the value function $v$ from below and any stochastic supersolution $w$ bounds $v$ from above:
\begin{equation} \label{eq:2.18}
u_-:=\sup_{u \in \mathcal V^-} u\le v\le  w_+:= \inf_{w \in \mathcal V^+} w\quad \textnormal{on } \overline\Pi.
\end{equation}

To see this put $\tau=0, \xi=x, \eta=y$ and $\rho=T^{x,y,\alpha}$ in Definition \ref{def:1}:
\begin{align*}
 Z_0^{x,y,\alpha}(u) =u(x,y) &\le\mathsf E\int_0^{T^{x,y,\alpha}} e^{-\beta s} f(X^{x,y,\alpha}_s,\alpha_s)\,ds \\
 &+ \mathsf E\left(e^{-\beta T^{x,y,\alpha}} u(X^{x,y,\alpha}_{T^{x,y,\alpha}},Y^{x,y,\alpha}_{T^{x,y,\alpha}})\right)\le v(x,y).
\end{align*}
The last inequality follows from the representation (\ref{eq:2.6}) since
$$e^{-\beta T^{x,y,\alpha}} u(X^{x,y,\alpha}_{T^{x,y,\alpha}},Y^{x,y,\alpha}_{T^{x,y,\alpha}})\le e^{-\beta T^{x,y,\alpha}}g(X^{x,y,\alpha}_{T^{x,y,\alpha}},Y^{x,y,\alpha}_{T^{x,y,\alpha}})=e^{-\beta T^{x,y,\alpha}}\psi(X^{x,y,\alpha}_{T^{x,y,\alpha}}).$$

Similarly, by Definition \ref{def:2} we have
\begin{align*}
Z_0^{x,y,\alpha}(w)=w(x,y) &\ge \mathsf E\int_0^{T^{x,y,\alpha}} e^{-\beta s} f(X^{x,y,\alpha}_s,\alpha_s)\,ds\\
 &+\mathsf E\left(e^{-\beta T^{x,y,\alpha}} w(X^{x,y,\alpha}_{T^{x,y,\alpha}},Y^{x,y,\alpha}_{T^{x,y,\alpha}})\right)
\end{align*}
for all $\alpha\in\mathcal U$. Thus $w(x,y)\ge v(x,y)$, since
$$e^{-\beta T^{x,y,\alpha}} w(X^{x,y,\alpha}_{T^{x,y,\alpha}},Y^{x,y,\alpha}_{T^{x,y,\alpha}})\ge e^{-\beta T^{x,y,\alpha}}g(X^{x,y,\alpha}_{T^{x,y,\alpha}},Y^{x,y,\alpha}_{T^{x,y,\alpha}})=e^{-\beta T^{x,y,\alpha}}\psi(X^{x,y,\alpha}_{T^{x,y,\alpha}}).$$

\begin{lemma} \label{lem:2}
Put $u(x,y)=\psi(x)$, $w(x,y)=\psi(x)+cy$. Under Assumptions \ref{as:1}, \ref{as:3} we have $u\in\mathcal V_-$, $w\in\mathcal V_+$ for $c>0$ large enough. Moreover, $\overline v$, defined by (\ref{eq:2.11}), is a stochastic supersolution under Assumption 2.
\end{lemma}
\begin{proof}
(i) Similar to (\ref{eq:2.8}), by Ito's formula and the definition of $\psi$ we get
\begin{equation} \label{eq:2.19}
Z_t^{\tau,\xi,\eta,0}(u)=\int^t_\tau e^{-\beta s} f(X_s^{\tau,\xi,\eta,0},0)\,ds+e^{-\beta t}\psi(X_t^{\tau,\xi,\eta,0})
=e^{-\beta\tau}\psi(\xi)+M_t
\end{equation}
on $\llbracket\tau,T^{\tau,\xi,\eta,0}\llbracket$, where $M_t=\int_\tau^t e^{-\beta s} \psi_x(X^{\tau,\xi,\eta,0})\cdot \sigma(X^{\tau,\xi,\eta,0}_s,0) \, dW_s$ is a local martingale. Since $M_\tau=0$ on $\{\tau<T^{\tau,\xi,\eta,0}\}$, and $M$ is bounded, as follows from (\ref{eq:2.19}), we have
$$ \mathsf E(Z_\rho^{\tau,\xi,\eta,0}(u)|\mathscr F_\tau) I_{\{\tau<T^{\tau,\xi,\eta,0}\}}=e^{-\beta\tau}\psi(\xi) I_{\{\tau<T^{\tau,\xi,\eta,0}\}}$$
for a stopping time $\rho$, satisfying the inequality
\begin{equation} \label{eq:2.20}
\tau\le\rho<T^{\tau,\xi,\eta,0}\quad \text{on }\{\tau<T^{\tau,\xi,\eta,0}\}.
\end{equation}
Moreover, since any $\mathbb F$-stopping time is predictable (see \cite[Proposition 16.22]{Bass11}), we may extend (\ref{eq:2.20}) to a stopping time $\rho\le T^{\tau,\xi,\eta,0}$ by the continuity argument.

It follows that $u$ is a stochastic subsolution:
\begin{equation} \label{eq:2.21}
\mathsf E(Z_\rho^{\tau,\xi,\eta,0}(u)|\mathscr F_\tau)=e^{-\beta\tau}\psi(\xi)=Z_\tau^{\tau,\xi,\eta,0}(u),\quad
\rho\in [\tau,T^{\tau,\xi,\eta,0}],
\end{equation}
since on $\{\tau=T^{\tau,\xi,\eta,0}\}$ this equality is trivially satisfied. Note, that the control process $\alpha=0$, ensuring (\ref{eq:2.21}), does not depend on $(\tau,\xi,\eta)$, although such dependence is allowed by Definition \ref{def:1}.

(ii) To prove that $w=\psi(x)+cy$ is a stochastic supersolution of (\ref{eq:2.1}), (\ref{eq:2.2}), we also apply Ito's formula:
\begin{align*}
Z^{\tau,\xi,\eta,\alpha}_t(w) &=e^{-\beta \tau}w(\xi,\eta)+N_t+M_t \quad \text{on } \llbracket\tau,T^{\tau,\xi,\eta,0}\llbracket,\\
N_t &=\int^t_\tau e^{-\beta s} f(X^{\tau,\xi,\eta,\alpha}_s,\alpha)\,ds\\ &+\int^t_\tau e^{-\beta s}\left(\mathcal L^\alpha \psi(X^{\tau,\xi,\eta,\alpha}_s) - \beta \psi(X^{\tau,\xi,\eta,\alpha}_s) -|\alpha|c - \beta c Y^{\tau,\xi,\eta,\alpha}_s\right)\,ds, \\
M_t &=\int^t_\tau e^{-\beta s} \psi_x(X^{\tau,\xi,\eta,\alpha}_s) \cdot \sigma(X^{\tau,\xi,\eta,\alpha}_s,\alpha) \, dW_s.
\end{align*}
By Assumption \ref{as:3} there exists a constant $K>0$ such that
$$|f(x,a)+\mathcal L^a\psi(x)-\beta\psi(x)| =|f(x,a)-f(x,0)+\mathcal L^a\psi(x)-\mathcal L^0\psi(x)|\le K|a|.$$
Hence,
\begin{align*}
 |N_t| &\le\int^t_\tau e^{-\beta t} (K|\alpha_s|+c|\alpha_s|+\beta c \eta)\,ds\le K',\\
  N_t  &\le\int^t_\tau e^{-\beta t} (K-c)|\alpha_s|\,ds\le 0\quad \text{for } c\ge K.
\end{align*}
It follows that the local martingale $M$ is uniformly bounded on $\llbracket\tau,T^{\tau,\xi,\eta,0}\llbracket$ and
$$ \mathsf E(Z^{\tau,\xi,\eta,\alpha}_\rho(w)|\mathscr F_\tau)I_{\{\tau<T^{\tau,\xi,\eta,\alpha}\}}\le e^{-\beta \tau}w(\xi,\eta)I_{\{\tau<T^{\tau,\xi,\eta,\alpha}\}}$$
for a stopping time $\rho$, satisfying (\ref{eq:2.20}). As in the proof of part (i), one can extend this inequality to a stopping time $\rho\le T^{\tau,\xi,\eta,\alpha}$ to obtain
$$\mathsf E(Z^{\tau,\xi,\eta,\alpha}_\rho(w)|\mathscr F_\tau)\le e^{-\beta \tau}w(\xi,\eta),$$
which means that $w$ is a stochastic supersolution.

(iii) By \cite{Rok14} (see Theorem 1 and Remark 1)  $\overline v\in C(\overline G)$ is the unique bounded viscosity solution of (\ref{eq:2.12}), (\ref{eq:2.13}), and it satisfies the boundary condition $\overline v=0$ on $\partial\overline G$. The argumentation of \cite{Rok14} shows that there exist a decreasing sequence of stochastic supersolutions $\overline w_n$ of (\ref{eq:2.12}), (\ref{eq:2.13}), converging to $\overline v$. From Definition \ref{def:2} it follows that $\overline v$ is a stochastic supersolution of (\ref{eq:2.12}), (\ref{eq:2.13}). Since $\overline v$ does not depend on $y$ and $\overline v\ge\psi$ (and thus $\overline v\ge g$ on $\partial\Pi$), from the same definition it follows that $\overline v$ a stochastic supersolution of (\ref{eq:2.4}), (\ref{eq:2.15}).
\end{proof}

\begin{proof}[Proof of Theorem \ref{th:1}]
By Lemma \ref{lem:2} there exist a pair $u$, $w$ of stochastic sub- and supersolutions such that
$$ u=\psi=w\quad \text{on } G\times\{0\},$$
and there exist another such pair $u$, $\overline v$ such that
$$ u=0=\overline v\quad \text{on } \partial G\times [0,\infty).$$
By the definition (\ref{eq:2.18}) of $u_-$, $w_+$ it follows that
\begin{equation} \label{eq:2.22}
u_- =w_+\quad \text{on } \partial\Pi.
\end{equation}

Furthermore, it was proved in \cite{Rok14} (Theorems 2, 3) that $u_-$ (resp., $w_+$) is a viscosity supersolution (resp., viscosity subsolution) of (\ref{eq:2.4}). By the comparison result (see \cite[Theorem 7.3]{Ish89}, \cite[Theorem 4.2]{MotSar08a}, \cite[Theorem 6.21]{Tou13} for the case of unbounded domain) and (\ref{eq:2.22}) we get the inequality $u_- \ge w_+$ on $\Pi$. Combining this inequality with (\ref{eq:2.18}), we conclude that
$$u_-=v=w_+ \quad \text{on} \quad \overline{\Pi}.$$

Hence, $v$ is a continuous on $\overline G$ viscosity solution of (\ref{eq:2.4}), which satisfies the boundary condition (\ref{eq:2.15}) in the usual sense. The uniqueness of a continuous viscosity solution follows from the same comparison result.
\end{proof}

\begin{remark}
Theorem \ref{th:1} can be adapted to the case of finite horizon $\widetilde T$. In this case one should consider the extended controlled process $\widetilde X=(t,X)$ and the domain $\widetilde G=(0,\widetilde T)\times G$ instead of $X$ and $G$. The HJB equation $\widetilde G\times (0,\infty)$ is analysed along the same lines. The correspondent parabolic equations (\ref{eq:2.5}), (\ref{eq:2.12}) can be considered as degenerate elliptic: see, e.g., \cite{Kry96} for linear case, and \cite[Corollary 3.1]{Cha04} for the strong comparison result in the nonlinear case. The latter is needed to make Assumption 2 constructive.
\end{remark}

\begin{remark} The Dirichlet condition $v(x,0)=\psi(x)$, $x\in \overline G$, where $\psi$ is the value function of the uncontrolled problem without fuel, is quite natural and is commonly used in finite fuel control problems (see, e.g., \cite[Sect. VIII.6]{FleSon06}). However, some authors use another condition, which is typical for state constrained problems: see \cite{PemZhaYin07,MotSar07}.
\end{remark}

\begin{remark} The idea of utilization of stochastic semisolutions, satisfying the sought-for boundary conditions (see Lemma \ref{lem:2}), in order to reduce the proof of Theorem \ref{th:1} to a standard comparison result, is borrowed from \cite{BayZha05}.
\end{remark}

\section{Optimal correction of a stochastic system} \label{sec:3}
\setcounter{equation}{0}
Consider a simple one-dimensional ($d=1$) controlled stochastic system
 \begin{align*}
dX &=-kXdt+\sigma dW_t-\alpha_t dt, \\
dY &=-|\alpha_t|dt,
\end{align*}
where $\sigma>0$, $k$ are some constants and $\alpha_t\in [\underline a,\overline a]$. The case $k>0$ (resp., $k<0$) corresponds to the stable (resp., unstable) equilibrium point $0$. An infinitesimal increment $dX$ of the system can be corrected with intensity $\alpha$.
Controller's aim is to keep the system in the interval $G=(-l,l)$, $l>0$ as long as possible. More precisely we consider the risk-sensitive criterion of the form (\ref{eq:1.4}). By Lemma \ref{lem:1} we pass to the exit time problem
$$ v(x,y)=\sup_{\alpha\in\mathcal U}\mathsf E\left(\int_0^{T^{x,y,\alpha}} e^{-\beta t}\,dt+e^{-\beta T^{x,y,\alpha}}\psi(X_{T^{x,y,\alpha}}^{x,y,\alpha})\right),$$
where $\psi$ is the solution of the Dirichlet problem for the ordinary differential equation:
\begin{align*}
&\beta \psi -1 +kx\psi_x-\frac{1}{2}\sigma^2 \psi_{xx}=0, \quad x \in (-l,l)\\
& \psi(-l)=\psi(l)=0.
\end{align*}

Clearly, Assumptions 1-3 holds true. By Theorem \ref{th:1} $v$ is the unique continuous bounded viscosity solution of the equation
\begin{equation} \label{eq:3.1}
\beta v -1-\frac{1}{2}\sigma^2 v_{xx}+\min_{a \in [\underline a,\overline{a}]}\{(kx+a)v_x+|a|v_y\}=0,\quad (x,y)\in (-l,l)\times(0,\infty)
\end{equation}
which satisfies the boundary conditions:
\begin{equation} \label{eq:3.2}
v(x,0)=\psi(x),\ x\in [-l,l];\quad v(-l,y)=v(l,y)=0,\  y>0.
\end{equation}

To solve the problem (\ref{eq:3.1}), (\ref{eq:3.2}) numerically we consider the rectangular grid
$$\overline G_h=\{x_{ij}=(ih_1,jh_2): -I \le i \le I, \ 0 \le j \le J\}, \ I h_1=l, \ Jh_2=\overline y,$$
where $I,J,i,j$ are integers, $h=(h_1,h_2)$ are the grid steps, and $\overline y$ corresponds to the artificial boundary.
Put $G_h=\{x_{ij}:-I<i<I, 0<j\le J\}$ and denote by $\partial G_h=\overline G_h\backslash G_h$ the "parabolic boundary" of the
grid. Consider the system of equations
\begin{align}
& \beta v_{i,j} -1-\frac{1}{2}\sigma^2 \frac{v_{i+1,j}-2v_{i,j}+v_{i-1,j}}{h^2_1}+\min_{a \in [\underline a,\overline a]} \left\{ (k x_{i,j}+a)^+\frac{v_{i,j}-v_{i-1,j}}{h_1}\right.\nonumber\\
 &-(k x_{i,j}+a)^-\frac{v_{i+1,j}-v_{i,j}}{h_1} \left.+|a|\frac{v_{i,j}-v_{i,j-1}}{h_2} \right\}=0,\quad x_{ij}\in G_h;\label{eq:3.3}\\
& v_{ij}-g(x_{ij})=0,\quad x_{ij}\in \partial G_h \label{eq:3.4}
\end{align}
for the mesh function $v_{ij}=v_h(x_{ij})$. The function $g$ is defined by (\ref{eq:2.14}).
Equations (\ref{eq:3.3}), (\ref{eq:3.4}) can be represented in the form
\begin{equation} \label{eq:3.5}
F_h(x_{ij},v_{ij},(v_{ij}-v_{i'j'})_{x_{i'j'}\in\Gamma(x_{ij})})=0,\ \ x_{ij}\in \overline G_h,
\end{equation}
where $\Gamma(x_{ij})$ is the set of neighbors of $x_{ij}$:
$$\Gamma(x_{ij})=\{x_{i+1,j},x_{i-1,j},x_{i,j-1}\},\ x_{ij}\in G_h;
\quad\Gamma(x_{ij})=\emptyset,\ x_{ij}\in \partial G_h.$$

The function $F_h$ is nondecreasing in each variable, except of $x_{ij}$. In the terminology of \cite{Obe06} the scheme (\ref{eq:3.5}) is \emph{degenerate elliptic}. The inequality
$$F_h(x,r,y)-F_h(x,r',y)=\beta' (r-r')>0, \ \ \beta'=\min\{\beta,1\}\ \ \text{for}\ r>r',$$
means, that scheme is \emph{proper} \cite{Obe06}.
Furthermore, the scheme is \emph{Lipshitz continuous} with constant $K_h$:
\begin{align}
& | F_h[x,z]-F_h[x,z']| \le K_h \|z-z'\|_{\infty},\label{eq:3.6}\\
K_h &=\max\{1,\beta\}+\frac{\sigma^2}{h_1^2}+\frac{|k|l+\overline a}{h_1}+\frac{\overline a}{h_2}.\nonumber
\end{align}
We use the notation $F_h[x,v_h]$ for the left-hand side of (\ref{eq:3.5}), and $\|z\|_\infty=\max\{|z_{ij}|:z_{ij}\in \overline G_h\}$. The proof of (\ref{eq:3.6}) is based on the elementary inequality
$$ |\max_{q\in Q}\phi(x,q)-\max_{q\in Q}\phi(y,q)|\le\max_{q\in Q} |\phi(x,q)-\phi(y,q)|, $$
which is valid for a continuous function $\phi$ and a compact set $Q$.

In \cite{Obe06} (Theorem 7) it is shown, that the operator $S_\rho(v)=v-\rho F_h[x,v]$ is a strict contraction in the space of mesh functions, equipped with norm $\|\cdot\|_\infty$, for sufficiently small $\rho$:
$$ \| S_\rho(u)-S_\rho(v)\|_\infty\le\gamma\|u-v\|_\infty,\ \ \gamma=\max\{1-\rho\beta',\rho K_h\}.$$
It follows that (\ref{eq:3.5}) has an unique solution $v_h$, which coincides with the fixed point of $S_\rho$, and it can be approximated by the iterations
\begin{equation} \label{eq:3.7}
v^{n+1}=S_\rho(v^n),\ \ \|v_h-v^n\|_\infty\to 0
\end{equation}
with an arbitrary $v^0$.

Theorem \ref{th:1} allows to justify the convergence of mesh functions $v_h$, $h\to 0$ to the value function $v$ by the Barles-Souganidis method \cite[Theorem 2.1]{BarSou91}. To do this one should check the \emph{stability}, \emph{monotonicity} and \emph{consistency} properties of the finite difference scheme (\ref{eq:3.5}) (see \cite{BarSou91} for the definitions).

The monotonicity property means that the function $F_h(x_{ij},v_{ij},(v_{ij}-v_{i'j'})_{x_{i'j'}\in\Gamma(x_{i,j}) })$ is nonincreasing in $v_{i'j'}$, and follows from the fact that the scheme is degenerate elliptic. Furthermore, since $0\le\psi\le 1/\beta$, we get the inequalities
$$ F_h[x,0]\le F_h[x,v_h]=0\le F_h[x,1/\beta],$$
which imply the stability property:  $0\le v_h\le 1/\beta$ by \cite[Theorem 5]{Obe06}. The proof of the consistency property is based on Taylor's formula: see, e.g., \cite{Tour13} for a simple example. We do not go in details here.

The computer experiments were performed for the following set of parameters: $\beta=0.1$, $\sigma=0.8$, $l=1$, $\overline y=40$, $\overline a=-\underline a=10$. To analyze the influence of the attraction (repulsion) rate on optimal strategies we considered the values $k \in [-10,10]$. The calculations were implemented on the $200\times 200$ grid, covering the rectangle $[-1,1]\times[0,40]$.

Choose $\rho=1/(2 K_h)$ in the iteration method (\ref{eq:3.7}) and take $\underline v^0=0$ and $\overline v^0= 1/ \beta$ as initial values.
It is easy to see that
$$ S_\rho(\underline v^0)-\underline v^0=-\rho F_h[x,\underline v^0]\ge 0,\quad
   S_\rho(\overline v^0)-\overline v^0=-\rho F_h[x,\overline v^0]\le 0.$$
From the monotonicity property of the operator $S_\rho$ (\cite[Theorem 6]{Obe06}) it follows that the iterations with these initial values converge monotonically:
$\underline v^n\uparrow v_h,$ $\overline v^n\downarrow v_h.$ The iterations were performed until
\begin{equation} \label{eq:3.8}
\max_{ij}({\overline v^n_{ij}-\underline v^n_{ij}})/\underline v^n_{ij}\le\varepsilon=0.01.
\end{equation}
Typically this required about $400$ thousand steps.

For $k=2$ the graph and the level sets of the value function are presented in Figure \ref{fig:1}. Clearly, $v$ is symmetric with respect to the axis $x=0$, and is increasing in $y$.
\begin{figure}[ht!]
        \centering
       \includegraphics[width=1\textwidth]{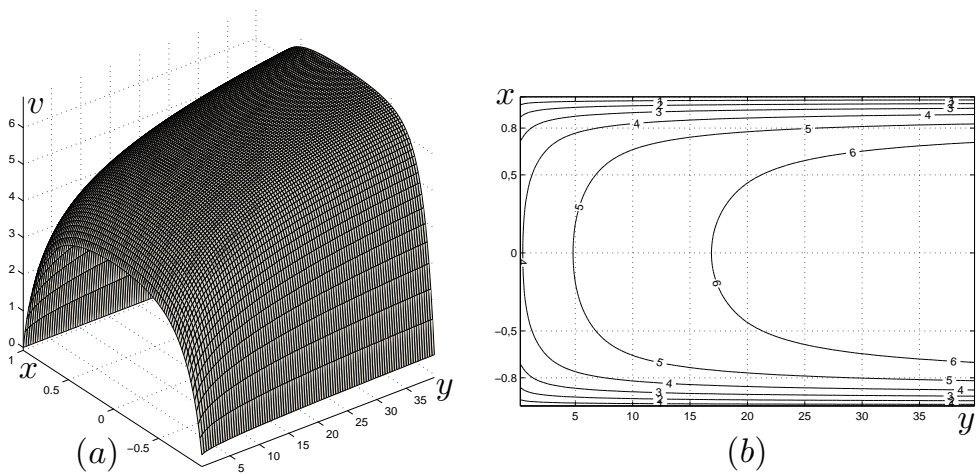}
         \caption{The value function (a) and its level sets (b) for $k=2$.}
         \label{fig:1}
\end{figure}

The switching lines of optimal strategies $\alpha^*$, corresponding to optimal values of $a$ in (\ref{eq:3.3}), are shown in Figure \ref{fig:2}. The middle area, containing the equilibrium, is the no-action region $G_{na}$, where $\alpha^*=0$. In the complement area we have $\alpha^*=-10$ near the upper boundary $x=l$, and $\alpha^*=10$ near the lower boundary $x=-l$.

\begin{figure}[ht!]
      \centering
          \includegraphics[width=1\textwidth]{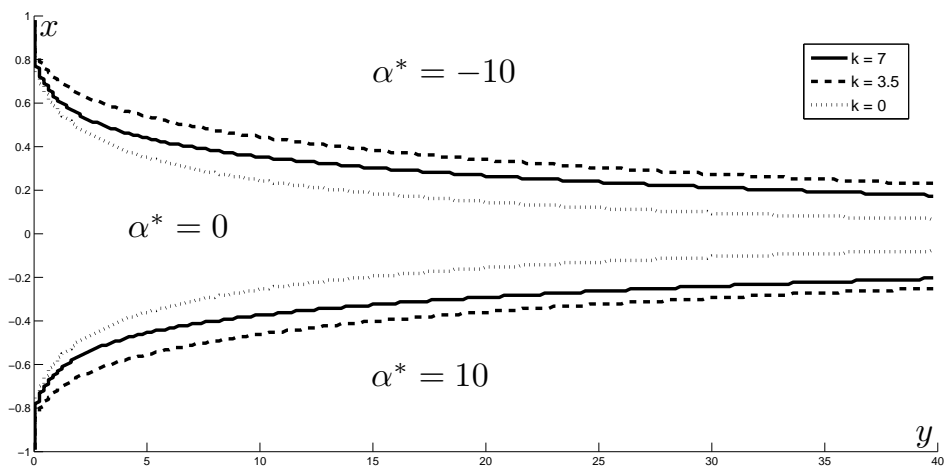}
        \caption{Optimal control in the stable case.}
        \label{fig:2}
\end{figure}

The no-action region widens when $y$ decreases. This means that the controller becomes less active when the recourse $Y$ runs low. More interesting and unexpected effect concerns the "non-monotonic" behavior of the no-action region with respect to $k$. It was detected experimentally that $G_{na}$ becomes wider, when $k$ grows from $0$ to $3.5$. Thus, the controller is less involved in the stabilization of the system, which becomes more stable itself. But, for $k>3.5$ we observe the opposite picture: the no-action region narrows as $k$ grows further!

Optimal strategies for the unstable case $k<0$ are presented in Figure 3. Here no-action regions are much smaller. It is not surprising since it is more difficult to keep the unstable system near the equilibrium point. In contrast to the stable case, here $G_{na}$ shrinks monotonically in $k$.
\begin{figure}[ht!]
        \centering
         \includegraphics[width=1\textwidth]{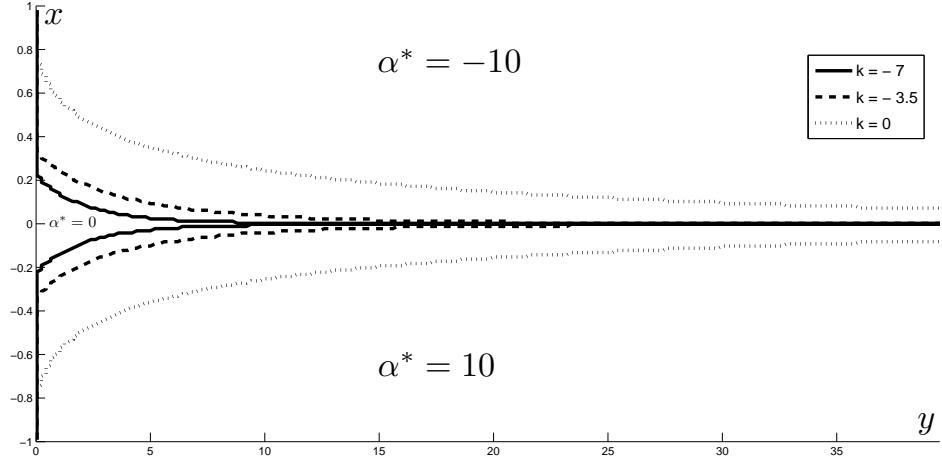}
        \caption{Optimal control in the unstable case.}
          \label{fig:3}
\end{figure}
Also, the value function $v$ is smaller. We do not present the graph of $v$, since it looks similar to Figure \ref{fig:1}(a).

\section{Optimal tracking of a stochastic system} \label{sec:4}
\setcounter{equation}{0}
Consider a random target $X^1$, which should be tracked by the controlled process $X^2$. The fluctuations of $X^1$ are described by the equation
\begin{align*}
dX^1_t &=\mu(X^1_t) dt +\sigma dW_t,\quad \sigma>0,\\
\mu(x_1)&=-k x_1 I_{\{|x_1|\le b\}} - k b I_{\{x_1\ge b\}}+k b I_{\{x_1\le -b\}},\quad b>0
\end{align*}
The case $k>0$ (resp., $k<0$) corresponds to the stable (resp., unstable) equilibrium point $0$ of the correspondent deterministic system. The dynamics of the tracking process $X^2$, controlled by the "fuel expenditure", is unaffected by noise:
$$ dX^2_t =\alpha_t dt,\quad dY_t =-|\alpha_t|dt,\quad \alpha_t\in [\underline a,\overline a].$$
We assume that the tracking is stopped if the target is "lost sight of":
$$\tau=\inf\{t\ge 0:|X^1_t- X^2_t|\ge l\},\quad l>0.$$

For the objective functional (\ref{eq:1.4}) the HJB equation (\ref{eq:2.4}) takes the form
$$ \beta v-1-\mu(x_1)v_{x_1}-\frac{1}{2}\sigma^2 v_{x_1 x_1}+\min_{a\in [\underline a,\overline a]}\{|a| v_y-a v_{x_2}\}=0,\quad
 (x,y)\in G\times (0,\infty), $$
where $G=\{x:|x_1-x_2|<l\}$. The boundary conditions (\ref{eq:2.15}) shapes to
$$ v=0\quad \text{on } \partial G\times [0,\infty);\quad v=\psi\quad \text{on } G\times\{0\},$$
where $\psi$ is the solution of the boundary value problem
\begin{align}
&\beta \psi-1-\mu(x_1)\psi_{x_1}-\frac{1}{2}\sigma^2 \psi_{x_1 x_1}=0,\quad x_1\in (x_2-l,x_2+l),\label{eq:4.1}\\
&\psi(x_2-l,x_2)=\psi(x_2+l,x_2)=0.\label{eq:4.2}
\end{align}

Let us check Assumptions 1-3. Let $\psi_1, \psi_2\in C^2(\mathbb R)$ be a fundamental solution system of the ordinary linear differential equation (\ref{eq:4.1}). Then
$$ \psi(x_1,x_2)=C_1(x_2)\psi_1(x_1)+C_2(x_2)\psi_2(x_1)+1/\beta, $$
where the $C_1$, $C_2$ are uniquely defined by the boundary conditions (\ref{eq:4.2}). It follows that $\psi\in C^2(G)$ and Assumption 1 is satisfied. Assumption 2 holds true since the condition (\ref{eq:2.13A}) is met. To verify Assumption 3 it is enough to show that $\psi$ and its derivatives up to second order are uniformly bounded in $G$. But this property follows the fact that for $|x|$ large enough, $\mu$ is constant, and we have $\psi=\varphi(x_1-x_2)$, where $\varphi(z)$ is defined by
$$\beta \varphi-1-\mu\varphi_z-\frac{1}{2}\sigma^2 \varphi_{zz}=0,\quad z\in (-l,l);\quad \varphi(-l)=\varphi(l)=0.$$

To solve the problem numerically, as in the previous example, we use the monotone finite difference scheme:
\begin{align*}
 0&=\beta v_{i,j,k} -1 -k\mu^+(x_{ijk})\frac{v_{i+1,j,k}-v_{i,j,k}}{h_1}+k\mu^+(x_{ijk})\frac{v_{i,j,k}-v_{i-1,j,k}}{h_1}\\
 &-\sigma^2\frac{v_{i-1,j,k}-2v_{i,j,k}+v_{i+1,j,k}}{2 h_1^2} \\
 &+\min_{a \in [\underline a,\overline a]}
 \left\{|a|\frac{v_{i,j,k}-v_{i,j,k-1}}{h_3}-a^+\frac{v_{i,j+1,k}-v_{i,j,k}}{h_2} +a^-\frac{v_{i,j,k}-v_{i,j-1,k}}{h_2}   \right\}, \quad x_{ijk}\in G_h;\\
0&=v_{ijk}-g(x_{ijk}),  \quad x_{ijk}\in \partial G_h.
\end{align*}
The grid $\overline G_h$ is the subset of points $\{x_{ijk}=(ih_1,jh_2,kh_3)\in\overline G(\overline x,\overline y):(i,j,k)\in \mathbb Z\times\mathbb Z\times\mathbb Z_+\}$. The set $\overline G(\overline x,\overline y)=\{|x_1-x_2|\le l,\ |x_1+x_2|\le \overline x,\ y\in [0,\overline y]\}$ is cut out from $\overline G\times [0,\infty]$. The values $\overline x$, $\overline y$ determine the artificial boundary. As in Section \ref{sec:3}, by $\partial G_h$ we denote the parabolic boundary of $\overline G_h$, that is, $\partial G_h$ contains all points of $\overline G_h\cap\partial \overline G(\overline x,\overline y)$, except of those with maximal values of $k$ index. Other points of the grid we attribute to $G_h$.

The scheme is analyzed along the same lines as in Section \ref{sec:3}. The convergence of its solution $v_h$ to the value function $v$ follows from Theorem \ref{th:1} by the Barles-Souganidis method. The grid function $v_h$ is obtained by the iterations (\ref{eq:3.7}). The initial approximations $\underline v^0=0$, $\overline v^0= 1/ \beta$ and the stopping criterion
(\ref{eq:3.8}) still apply.

In experiments we used the following parameters: $\beta=0.1$, $\sigma=0.8$, $b=2.5$, $l=4/\sqrt 2$, $\overline y=10$, $\overline x=80$, $\overline a=-\underline a=1$. The grid contained $2000\times 100\times 50$ nodes $x_{ijk}$. With $\varepsilon=0.01$ in (\ref{eq:3.8}), iterations typically stopped after 10 thousand steps.

It is convenient to make the rotation transform
$$z_1=(x_1+x_2)/\sqrt 2,\quad z_2=(x_1-x_2)/\sqrt 2$$
and present the results in the new variables $(z_1,z_2)$. The switching lines of optimal control in the stable ($k=0.3$) and unstable ($k=-0.3$) cases are shown in Figures \ref{fig:4} and \ref{fig:5} respectively (for $y=1$).
\begin{figure}[ht!]
        \centering
       \includegraphics[width=1\textwidth]{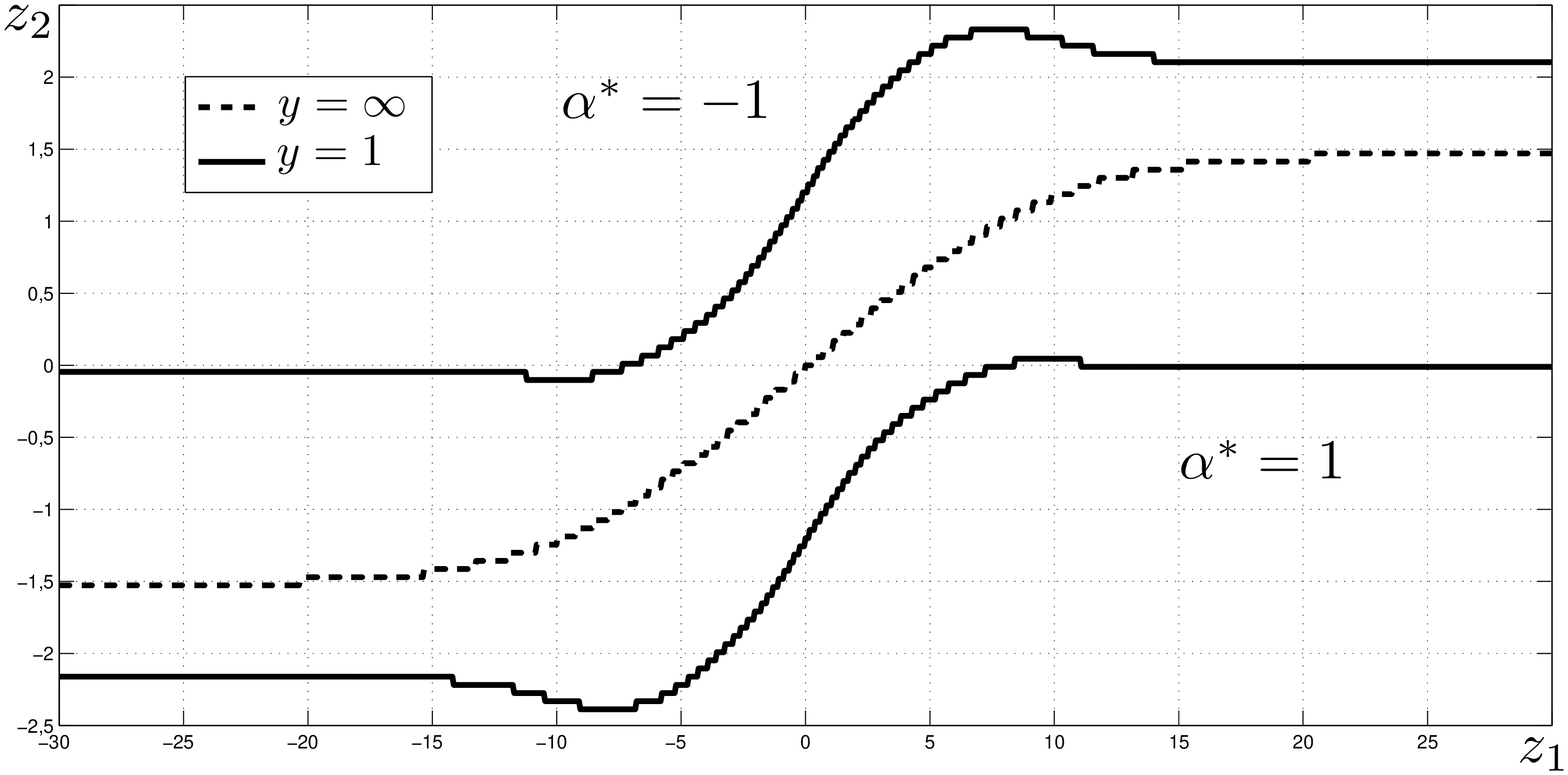}
         \caption{Optimal control in the stable case, $k=0.3$.}
         \label{fig:4}
\end{figure}
\begin{figure}[ht!]
        \centering
       \includegraphics[width=1\textwidth]{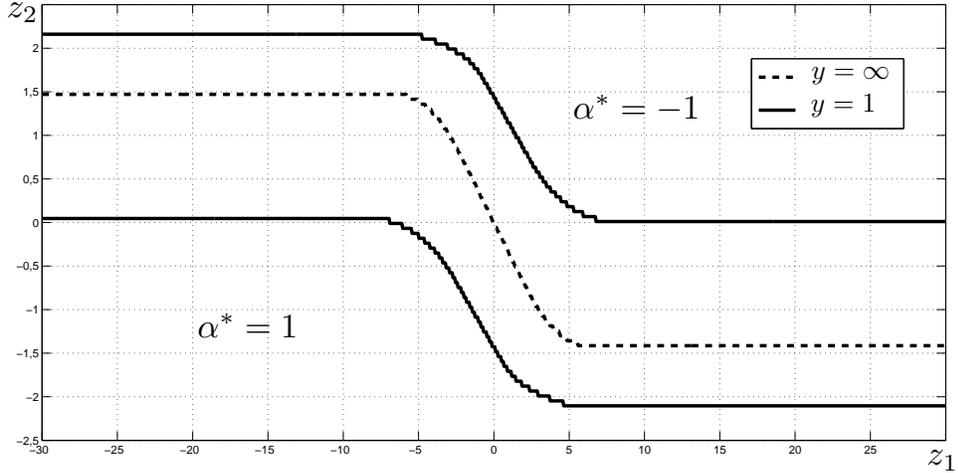}
         \caption{Optimal control in the unstable case, $k=0.3$.}
         \label{fig:5}
\end{figure}

The domains between solid lines correspond to the no-action sets. The dashed lines determine the switching of optimal control for the problem (\ref{eq:2.12}), (\ref{eq:2.13}) with infinite fuel (here the no-action region is empty).
Since $\mu$ is constant for $|x_1|>b$, the switching lines stabilize for $|z_1|$ large enough. Moreover, for large $z_1>0$ the no-action region is located above (resp., below) the line $z_2=0$ in the stable (resp., unstable) case. The reason is that in the stable case, for $\alpha=0$, the point $(Z^1_t,Z^2_t)=(X_t^1+X_0^2,X_t^1-X_0^2)/\sqrt 2$ with large $Z^1_0>0$, on average, goes from the upper boundary of the strip $(z_1,z_2)\in\mathbb R\times (-l,l)$ to its lower boundary. In the unstable case there is an opposite trend. For $z_1<0$ the pictures can be recovered by reflection with respect to the origin.

Examples of graphs and level sets of the value functions $v$ in the stable and unstable cases are given in Figures \ref{fig:6} and \ref{fig:7}.

\begin{figure}[ht!]
        \centering
       \includegraphics[width=1\textwidth]{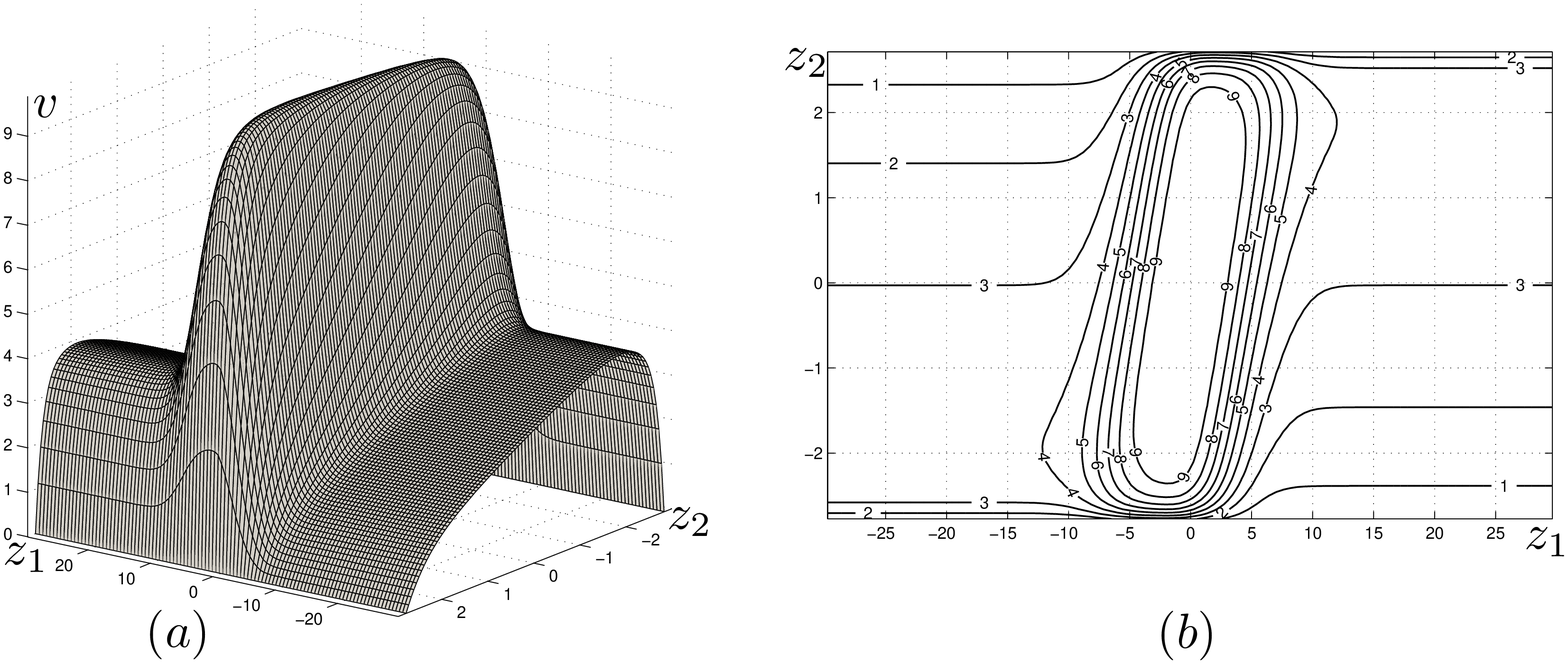}
         \caption{The value function (a) and its level set (b) for $k=0.3$.}
         \label{fig:6}
\end{figure}
\begin{figure}[ht!]
        \centering
       \includegraphics[width=1\textwidth]{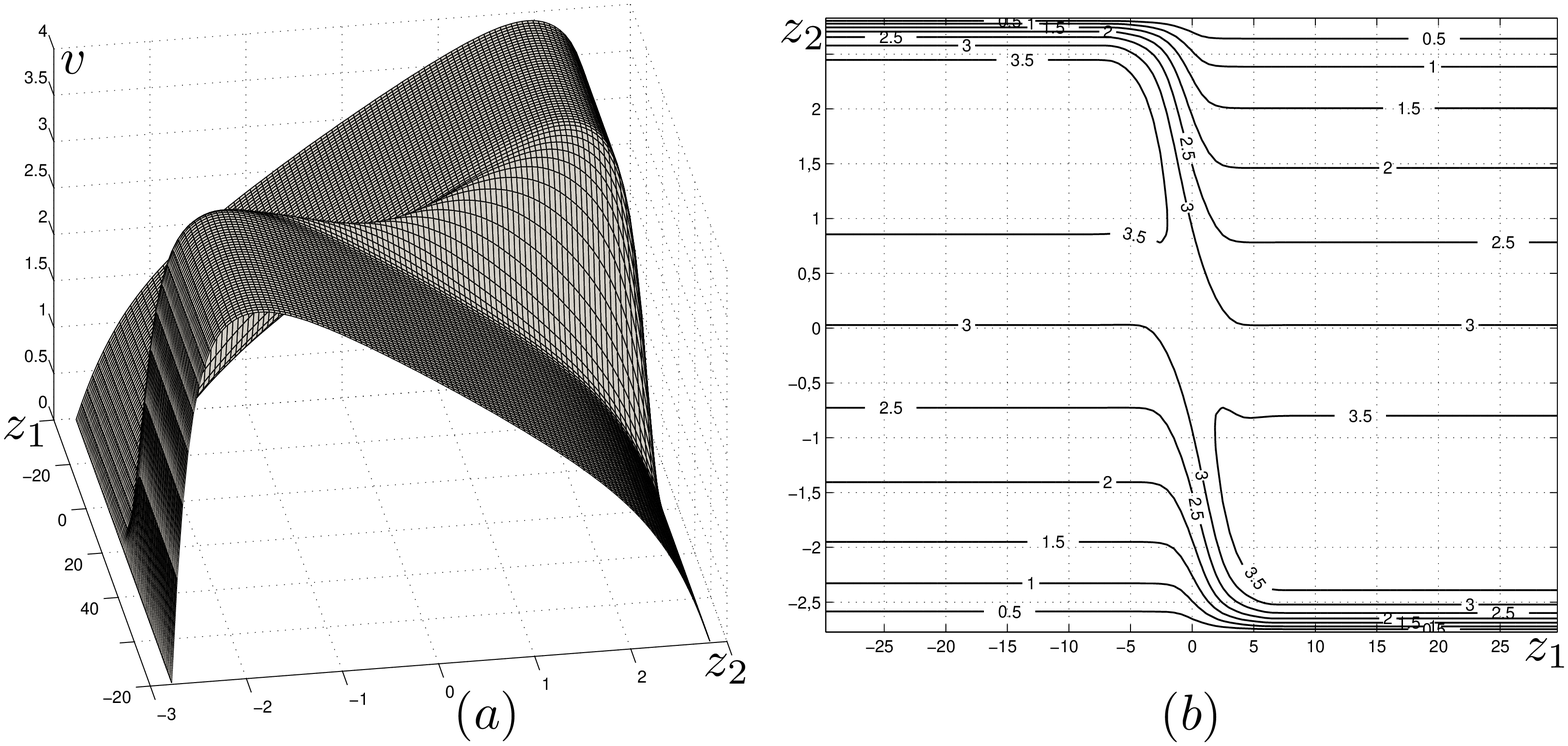}
         \caption{The value function (a) and its level set (b) for $k=-0.3$.}
         \label{fig:7}
\end{figure}

For fixed $z_1, y$ the value functions attain their maximum in $z_2$ in the no-action regions. The global maximum in the stable case is attained at the origin. However, in the unstable case, the maximum points of $v$ are located in the those parts of no-action sets, where $|z_1|$ is large and $v$ is approximately constant in $z_1$.



  \bibliographystyle{plain}
  \bibliography{litFuel}





\end{document}